\pdfoutput=1
\documentclass{amsart}

\usepackage{amsmath}             
\usepackage{amscd}
\usepackage{amssymb}             
\usepackage{amsfonts}            
\usepackage{latexsym}            
\usepackage{amsthm}              

\usepackage{enumerate}

\newcommand{\Z}{\mathbb{Z}}         
            
\newcommand{\Sp}{\operatorname{Sp}}
\newcommand{\SL}{\operatorname{SL}}
\newcommand{\GL}{\operatorname{GL}}
\newcommand{\SO}{\operatorname{SO}}
\newcommand{\rank}{\operatorname{rank}}

\newtheorem{lause}{Theorem}[section]
\newtheorem{lemma}[lause]{Lemma}
\newtheorem{seur}[lause]{Corollary}
\newtheorem{prop}[lause]{Proposition}
\newtheorem{prob}[lause]{Problem}
\newtheorem*{lause*}{Theorem}

\theoremstyle{definition}

\newtheorem{esim}[lause]{Example} 

\theoremstyle{remark}
\newtheorem{remark}[lause]{Remark}

\numberwithin{equation}{section}

\begin{document}

\title[]{Jordan blocks of unipotent elements in some irreducible representations of classical groups in good characteristic}

\author{Mikko Korhonen}
\address{School of Mathematics, The University of Manchester, Manchester M13 9PL, United Kingdom}
\email{korhonen\_mikko@hotmail.com}
\thanks{Some of the results in this paper were obtained during my doctoral studies at \'Ecole Polytechnique F\'ed\'erale de Lausanne, supported by a grant from the Swiss National Science Foundation (grant number $200021 \_ 146223$).}
%\subjclass[2010]{Primary 20G05}
\copyrightinfo{}{}
\date{\today}

\begin{abstract}

Let $G$ be a classical group with natural module $V$ over an algebraically closed field of good characteristic. For every unipotent element $u$ of $G$, we describe the Jordan block sizes of $u$ on the irreducible $G$-modules which occur as composition factors of $V \otimes V^*$, $\wedge^2(V)$, and $S^2(V)$. Our description is given in terms of the Jordan block sizes of the tensor square, exterior square, and the symmetric square of $u$, for which recursive formulae are known.
\end{abstract}

\vspace*{-1.5cm}
\maketitle

\section{Introduction}

Let $G$ be a simple linear algebraic group over an algebraically closed field $K$ of characteristic $p > 0$, and let $f: G \rightarrow \GL(W)$ be a rational irreducible representation. In this paper, we consider the following question in some special cases.

\begin{prob}\label{prob:unipjordans}
Let $u \in G$ a unipotent element. What is the Jordan normal form of $f(u)$?
\end{prob}

A basic motivation for Problem \ref{prob:unipjordans} is in the problem of determining the classes of unipotent elements that are contained in maximal subgroups of simple algebraic groups. This question is relevant in the study of the subgroup structure of simple algebraic groups, and solutions to specific instances of Problem \ref{prob:unipjordans} have found many applications. For example, when $G$ is simple of exceptional type, computations done by Lawther \cite{Lawther, LawtherCorrection} show that in most cases, the conjugacy class of a unipotent element of $G$ is determined by its Jordan block sizes in the adjoint and minimal modules of $G$. This was used in \cite{LawtherFusion} to determine the fusion of unipotent classes in maximal subgroups of $G$, which in turn was applied in \cite{BurnessLiebeckShalev} to a conjecture of Cameron on minimal base sizes of finite almost simple primitive groups. 

The origin of this paper is in the PhD thesis of the present author, which concerns the problem of finding reductive subgroups of simple algebraic groups that contain distinguished unipotent elements. For classical groups in good characteristic, this leads naturally to the question of when all Jordan block sizes of $f(u)$ have multiplicity one \cite[Proposition 3.5]{LiebeckSeitzClass}. The results of this paper have been useful in settling this problem in some important special cases.

Let $G$ be a simple classical group ($\SL(V)$, $\Sp(V)$, or $\SO(V)$) and assume that $p$ is \emph{good for $G$}. In other words, we assume $p > 2$ if $G = \Sp(V)$ or $G = \SO(V)$. Let $f: G \rightarrow \GL(W)$ be a rational irreducible representation with highest weight $\lambda$. As the main result of this paper, we describe the Jordan normal form of $f(u)$ for every unipotent element $u \in G$ in the following cases:

\begin{itemize}
\item $G = \SL(V)$ and $\lambda = \varpi_1 + \varpi_{n-1}$, where $n = \dim V$ (Theorem \ref{thm:mainthmA}).
\item $G = \Sp(V)$ and $\lambda = \varpi_2$ (Corollary \ref{corollary:SPmain}).
\item $G = \SO(V)$ and $\lambda = 2 \varpi_1$ (Corollary \ref{corollary:SOmain}).
\end{itemize}

These irreducible representations are found as composition factors of the tensor product $V \otimes V^*$ (Lemma \ref{lemma:typeAomega}), the exterior square $\wedge^2(V)$ (Lemma \ref{lemma:typeComega}), and the symmetric square $S^2(V)$ (Lemma \ref{lemma:typeBDomega}), respectively. Our answer is given in terms of the Jordan block sizes of the action of $u$ on $V \otimes V^*$, $\wedge^2(V)$, and $S^2(V)$, for which recursive formulae are known --- see Section \ref{section:tensordecomp}. Combining our main results with such formulae, by \cite[Theorem 5.1]{Lubeck} and Lemma \ref{lemma:typeAomega} -- \ref{lemma:typeBDomega} we have a description of the Jordan normal form of $f(u)$ for almost all irreducible representations $f$ of $G$ with $\dim W \leq (\rank G)^3 / 8$.

The rest of the paper is structured as follows. In Section \ref{section:notation} we establish notation used in the paper, and in Section \ref{section:prelim} we list some well known results that are needed in the paper. The proofs of our main results are based on calculations done in Section \ref{section:calculations}, where we consider the action of a regular unipotent element of $\SL(V)$ on $V \otimes V^*$. The main results are proven in Section \ref{section:mainresult}. 

\begin{remark}In characteristic zero, Problem \ref{prob:unipjordans} has a satisfactory solution. By results of Jacobson-Morozov-Kostant, every unipotent element $u \in G$ can be embedded into a simple subgroup $X < G$ of type $A_1$, which is unique up to conjugacy in $G$ \cite[Theorem 3.6]{KostantSL2}. Furthermore, by Weyl's complete reducibility theorem, the Jordan block sizes of $u$ on $W$ are determined by the character of the restriction of $W$ to $X$. This character can be computed with the Weyl character formula and the labeled Dynkin diagram associated with $u$. 

In this paper we are concerned with the positive characteristic case, where much less is known. In general we do not even know the dimensions of the rational irreducible representations of $G$, so we are still very far from a complete solution. One general result is in \cite{SuprunenkoMin}, where Suprunenko determines the largest Jordan block size of $f(u)$ when $G$ is simple of classical type and $p > 2$. More specialized results are found for example in \cite{PremetSuprunenkoQuadratic}, \cite[Section 2.3]{TiepZalesski}, and \cite{OsinovskayaSuprunenko}.\end{remark}

\section{Notation}\label{section:notation}

We fix the following notation and terminology. Throughout the text, let $K$ be an algebraically closed field of characteristic $p > 0$, let $G$ be a simple algebraic group over $K$, and let $V$ be a vector space over $K$.

Fix a maximal torus $T$ of $G$ with character group $X(T)$, and a base $\Delta = \{ \alpha_1, \ldots, \alpha_l \}$ for the root system of $G$, where $l = \operatorname{rank} G$. Here we use the standard Bourbaki labeling of the simple roots $\alpha_i$, as given in \cite[11.4, p. 58]{Humphreys}. We denote the dominant weights with respect to $\Delta$ by $X(T)^+$, and the fundamental dominant weight corresponding to $\alpha_i$ is denoted by $\varpi_i$. For a dominant weight $\lambda \in X(T)^+$, we denote the rational irreducible $G$-module with highest weight $\lambda$ by $L_G(\lambda)$.

If a $G$-module $V$ has a filtration $V = V_1 \supset V_2 \supset \cdots \supset V_{t} \supset V_{t+1} = 0$ with $\operatorname{soc}(V /V_{i+1}) = V_i / V_{i+1} \cong W_i$ for all $1 \leq i \leq t$, we will denote this by $V = W_1 | W_2 | \cdots | W_t$. For a $K$-vector space $V$ and non-negative integer $n$, we use the notation $n \cdot V$ for the direct sum $V \oplus \cdots \oplus V$, where $V$ occurs $n$ times.

For non-negative integers $a$ and $b$ we denote by $\binom{a}{b}$ the usual binomial coefficient, using the convention that $\binom{a}{b} = 0$ if $a < b$. We denote by $\nu_p$ the $p$-adic valuation on the integers, so $\nu_p(a)$ is the largest integer $k \geq 0$ such that $p^k$ divides $a$.

Let $u \in \GL(V)$ be unipotent. For all $m \geq 1$, we denote by $r_m(u)$ the number of Jordan blocks of size $m$ in the Jordan decomposition of $u$.

\section{Preliminaries}\label{section:prelim}

In this section, we list some preliminaries needed in this paper. All of the results in this section are well known. For calculations, we will need the following on the values of binomial coefficients modulo a prime $p$.

\begin{lause*}[Lucas' theorem]
Let $p$ be a prime number and let $a$ and $b$ be non-negative integers. Write $a = \sum_{k = 0}^n a_k p^k$ and $b = \sum_{k = 0}^n b_k p^k$ for integers $0 \leq a_k, b_k \leq p-1$. Then $$\binom{a}{b} \equiv \binom{a_0}{b_0} \binom{a_1}{b_1} \cdots \binom{a_n}{b_n} \mod{p},$$ and in particular $\binom{a}{b} \equiv 0 \mod{p}$ if and only if $a_k < b_k$ for some $0 \leq k \leq n$.
\end{lause*}

\begin{proof}A short proof can be found in \cite{FineBinomial}.\end{proof}

\begin{lemma}\label{lemma:pminus1modp}
Let $p$ be a prime number. Then $\binom{p - 1}{t} \equiv (-1)^t \mod {p}$ for all $0 \leq t \leq p-1$.
\end{lemma}

\begin{proof}We proceed by induction on $t$. For $t = 0$ the claim is obvious, and for $0 < t \leq p-1$ the claim follows by induction, since $\binom{p-1}{t} = \binom{p}{t} - \binom{p-1}{t-1}$ and $\binom{p}{t} \equiv 0 \mod{p}$.\end{proof}

We shall also need the following basic results on the number of Jordan blocks of unipotent elements.

\begin{lemma}\label{jordanblockformula}
Let $u \in \GL(V)$ be unipotent and denote $X = u - 1$. Then for all $m \geq 1$, we have $r_m(u) = 2 \dim \operatorname{Ker} X^m - \dim \operatorname{Ker} X^{m+1} - \dim \operatorname{Ker} X^{m-1}.$
\end{lemma}
\begin{proof}The formula follows from the fact that $\dim \operatorname{Ker} X^m = \sum_{k = 1}^m t_k$, where $t_k$ is the number of Jordan blocks of size $\geq k$ \cite[1.1]{JantzenNilpotent}.\end{proof}

\begin{lemma}\label{jordanrestriction}
Let $u \in \GL(V)$ be unipotent and denote $X = u - 1$. Suppose that $W \subseteq V$ is a subspace invariant under $u$ such that $\dim V/W = 1$. Let $m \geq 0$ be such that $\operatorname{Ker} X^m \subseteq W$ and $\operatorname{Ker} X^{m+1} \not\subseteq W$. Then 

\begin{enumerate}[\normalfont (a)]

\item if $m = 0$, we have

\begin{itemize}
\item $r_1(u_W) = r_1(u) - 1$,
\item $r_i(u_W) = r_i(u)$ for all $i \neq 1$.
\end{itemize}

\item if $m \geq 1$, we have

\begin{itemize}
\item $r_{m+1}(u_W) = r_{m+1}(u) - 1$,
\item $r_m(u_W) = r_m(u) + 1$,
\item $r_i(u_W) = r_i(u)$ for all $i \neq m, m+1$.
\end{itemize}

\end{enumerate}

\end{lemma}

\begin{proof}We have $\operatorname{Ker} X^i \subseteq W$ for all $0 \leq i \leq m$, so $\dim \operatorname{Ker} X_W^i = \dim \operatorname{Ker} X^i$ for all $0 \leq i \leq m$. By Lemma \ref{jordanblockformula}, this implies that $r_i(u) = r_i(u_W)$ for all $1 \leq i \leq m-1$. 

Next note that $\operatorname{Ker} X^i \not\subseteq W$ for all $i \geq m+1$, which means that $V / W = (\operatorname{Ker} X^i + W) / W \cong W / \operatorname{Ker} X^i \cap W$. Hence $\dim \operatorname{Ker} X_W^i = \dim \operatorname{Ker} X^i - 1$ for all $i \geq m+1$. Thus by Lemma \ref{jordanblockformula}, we have $r_m(u_W) = r_m(u) + 1$ if $m \geq 1$. Similarly $r_{m+1}(u_W) = r_{m+1}(u) - 1$, and $r_i(u_W) = r_i(u)$ for all $i \geq m+2$.
\end{proof}

The following lemmas are used to construct the irreducible representations that our main results (Theorem \ref{thm:mainthmA}, Corollary \ref{corollary:SPmain}, Corollary \ref{corollary:SOmain}) are concerned with.

\begin{lemma}[{\cite[1.14]{SeitzClassical}, \cite[Proposition 4.6.10]{McNinch}}]\label{lemma:typeAomega}
Let $G = \SL(V)$, where $\dim V = n$ for some $n \geq 2$. Then as $G$-modules, we have $S^2(V) \cong L_G(2 \varpi_1)$ (if $p > 2$), $\wedge^2(V) \cong L_G(\varpi_2)$ (if $n > 2$), and $$V \otimes V^* \cong \begin{cases} L_G(\varpi_1 + \varpi_{n-1}) \oplus L_G(0) &\mbox{if } p \nmid n, \\
L_G(0) | L_G(\varpi_1 + \varpi_{n-1}) | L_G(0) \text{ } & \mbox{if } p \mid n. \end{cases}$$
\end{lemma}

\begin{lemma}[{\cite[1.14, 8.1 (c)]{SeitzClassical}, \cite[Proposition 4.2.2, Lemma 4.8.2]{McNinch}}]\label{lemma:typeComega}
Assume $p > 2$, and let $G = \Sp(V)$, where $\dim V = n$ for some $n \geq 4$. Then as $G$-modules, we have $S^2(V) \cong L_G(2 \varpi_1)$, and $$\wedge^2(V) \cong \begin{cases} L_G(\varpi_2) \oplus L_G(0) &\mbox{if } p \nmid n, \\
L_G(0) | L_G(\varpi_2) | L_G(0) \text{ } & \mbox{if } p \mid n. \end{cases}$$
\end{lemma}

\begin{lemma}[{\cite[8.1 (a) -- (b)]{SeitzClassical}, \cite[Proposition 4.2.2, Lemma 4.7.3]{McNinch}}]\label{lemma:typeBDomega}
Assume $p > 2$, and let $G = \SO(V)$, where $\dim V = n$ for some $n \geq 5$. Then as $G$-modules, we have $\wedge^2(V) \cong L_G(\varpi_2)$, and $$S^2(V) \cong \begin{cases} L_G(2 \varpi_1) \oplus L_G(0) &\mbox{if } p \nmid n, \\
L_G(0) | L_G(2 \varpi_1) | L_G(0) \text{ } & \mbox{if } p \mid n.\end{cases}$$
\end{lemma}

\begin{lemma}\label{lemma:typeAomegabasis}
Let $G = \SL(V)$ and set $n = \dim V$. Let $e_1, \ldots, e_n$ be a basis of $V$ and let $e_1^*$, $\ldots$, $e_n^*$ be the corresponding dual basis of $V^*$. Then $\sum_{i = 1}^n e_i \otimes e_i^*$ spans the unique $1$-dimensional $G$-submodule of $V \otimes V^*$.
\end{lemma}

\begin{proof}It is well known that there is a canonical isomorphism $V \otimes V^* \rightarrow \operatorname{End}(V)$ of $G$-modules. Furthermore, this isomorphism maps $\sum_{i = 1}^n e_i \otimes e_i^*$ to the identity map $V \rightarrow V$. A straightforward exercise in linear algebra shows that $\operatorname{End}(V)$ has a unique $1$-dimensional $G$-submodule spanned by the identity map, so the lemma follows.\end{proof}

\section{Jordan block sizes in tensor squares}\label{section:tensordecomp}

In this section, we briefly discuss the Jordan decomposition of tensor products, exterior squares, and symmetric squares of unipotent linear maps. In positive characteristic, it is convenient to describe results on this topic in terms of the representation theory of a cyclic $p$-group. 

Let $u$ be a generator of a cyclic $p$-group of order $q$. We have an obvious correspondence between unipotent linear maps of order at most $q$ and $K[u]$-modules, where the decomposition of a $K[u]$-module into indecomposable summands corresponds to the Jordan normal form. There exist exactly $q$ indecomposable $K[u]$-modules up to isomorphism, which we label by $V_1$, $\ldots$, $V_q$, where $\dim V_i = i$ and $u$ acts on $V_i$ as a full $i \times i$ Jordan block. 

Taking tensor products of $K[u]$-modules is an additive functor. Furthermore, $$\wedge^2(W \oplus W') \cong \wedge^2(W) \oplus (W \otimes W') \oplus \wedge^2(W')$$ and $$S^2(W \oplus W') \cong S^2(W) \oplus (W \otimes W') \oplus S^2(W')$$ for all $K[u]$-modules $W$ and $W'$. Thus for any $K[u]$-module $V$, the question of decomposing $V \otimes V$, $\wedge^2(V)$, and $S^2(V)$ into indecomposable summands is easily reduced to the problem of decomposing $V_m \otimes V_n$, $\wedge^2(V_n)$, and $S^2(V_n)$ for integers $m,n > 0$.

There is a large amount of literature concerning the decomposition of $V_m \otimes V_n$ into a direct sum of indecomposables, for example \cite{Srinivasan}, \cite{Ralley}, \cite{McFall}, \cite{Renaud}, \cite{Norman}, \cite{Norman2}, \cite{Hou}, and \cite{Barry}. In positive characteristic, we do not have an easy explicit decomposition of $V_m \otimes V_n$ as in characteristic $0$, but there are various recursive descriptions in terms of $m$ and $n$. For example, if $m,n < p^\alpha$ and $m+n > p^{\alpha}$, then $V_m \otimes V_n \cong (m+n-p^\alpha) V_{p^\alpha} \oplus (V_{p^\alpha-m} \otimes V_{p^{\alpha}-n}).$ 

One convenient recursive description of $V_m \otimes V_n$ is given in \cite[Theorem 1]{Barry}, and the same paper also describes $\wedge^2(V_n)$ and $S^2(V_n)$ when $p > 2$ \cite[Theorem 2]{Barry}. In characteristic $p = 2$, the decomposition of $\wedge^2(V_n)$ and $S^2(V_n)$ is found in \cite[Theorem 2]{GowLaffey} and \cite[Theorem 1.1, Theorem 1.2]{HimstedtSymonds}.

The following theorem was first proven in \cite{GlasbyPraegerXia}. We present a different proof using induction and restriction of $K[u]$-modules.

\begin{lause}[{\cite[Theorem 5]{GlasbyPraegerXia}}]\label{thm:GPX}
Let $\alpha \geq 0$ and $0 < m,n \leq q/p^{\alpha}$. Suppose that $V_m \otimes V_n \cong V_{d_1} \oplus \cdots \oplus V_{d_t}$. Then $V_{p^{\alpha}m} \otimes V_{p^{\alpha}n} \cong p^{\alpha} \cdot (V_{p^{\alpha}d_1} \oplus \cdots \oplus V_{p^{\alpha}d_t})$.
\end{lause}

\begin{proof}
Set $G = \langle u \rangle$ and let $H = \langle u^{p^{\alpha}} \rangle$. Note that \begin{equation}\label{eq:multiplesalphatensor}\operatorname{Ind}_H^G (V_d) \cong V_{p^{\alpha}d}\end{equation} for all $0 < d \leq p^{\alpha}$. Thus $V_{p^{\alpha}m} \otimes V_{p^{\alpha}n} \cong \operatorname{Ind}_H^G (V_m)\otimes V_{p^{\alpha}n} \cong \operatorname{Ind}_H^G(V_m \otimes \operatorname{Res}_H^G(V_{p^{\alpha}n}))$ by \cite[Lemma 5 (5), p. 57]{Alperin}. Since $\operatorname{Res}_H^G(V_{p^{\alpha}n}) \cong p^{\alpha} \cdot V_n$, we conclude that $V_{p^{\alpha}m} \otimes V_{p^{\alpha}n} \cong p^{\alpha} \cdot \operatorname{Ind}_H^G(V_m \otimes V_n)$. Now the theorem follows from~\eqref{eq:multiplesalphatensor}.\end{proof}

We finish this section with two lemmas which describe the smallest Jordan block size in the tensor square of a unipotent matrix.

\begin{lemma}\label{lemma:smallestonsingle}
Let $0 < n \leq q$ and set $\alpha = \nu_p(n)$. Then the smallest Jordan block size in $V_n \otimes V_n$ is $p^{\alpha}$, which occurs with multiplicity $p^{\alpha}$.
\end{lemma}

\begin{proof}Since $p$ does not divide $n/p^{\alpha}$, it follows from \cite[Theorem 2.1]{BensonCarlson} that $V_1$ is a direct summand of $V_{n/p^\alpha} \otimes V_{n/p^\alpha}$, occurring with multiplicity one. Now the lemma follows by applying Theorem \ref{thm:GPX} to $V_{n/p^\alpha} \otimes V_{n/p^\alpha}$.\end{proof}

\begin{lemma}\label{lemma:smallestontsquare}
Let $V = V_{d_1} \oplus \cdots \oplus V_{d_t}$, where $t \geq 1$ and $0 < d_i \leq q$ for all $i$. Let $\alpha = \nu_p(\gcd(d_1, \ldots, d_t))$. Then the smallest Jordan block size occurring in $V \otimes V$ is $p^\alpha$, which occurs with multiplicity at least $p^{\alpha}$.
\end{lemma}

\begin{proof}We have $V \otimes V \cong \bigoplus_{1 \leq i,j \leq t} V_{d_i} \otimes V_{d_j}$ as $K[u]$-modules. Since $p^\alpha$ divides $d_i$ and $d_j$ for all $1 \leq i,j \leq t$, by Theorem \ref{thm:GPX} every Jordan block size in $V_{d_i} \otimes V_{d_j}$ is a multiple of $p^{\alpha}$, so in particular at least $p^{\alpha}$. Furthermore, there exists some $i$ such that $\nu_p(d_i) = \alpha$, so the claim follows from Lemma \ref{lemma:smallestonsingle}.\end{proof}

\section{Action of a regular unipotent element of $\SL(V)$ on $V \otimes V^*$}\label{section:calculations}

For this section, we fix a basis $e_1$, $\ldots$, $e_n$ of $V$. Let $e_1^*$, $\ldots$, $e_n^*$ be the corresponding dual basis of $V^*$, so $e_i^*(e_j) = \delta_{i,j}$ for all $1 \leq i,j \leq n$. For convenience of notation, we set $e_i = 0$ and $e_i^* = 0$ for all $i \leq 0$ and $i > n$. Throughout this section, we denote by $u$ the unipotent linear map which is a single $n \times n$ Jordan block with respect to the basis $(e_i)$, that is, $$ue_i = e_i + e_{i-1}$$ for all $1 \leq i \leq n$. We denote by $X$ the element $u-1$ of $K[u]$. The main purpose of this section is to establish various formulae for the action of powers of $X$ on the $\SL(V)$-module $V \otimes V^*$. We begin with the following lemma.

\begin{lemma}\label{lemma:Xpower_basic}Let $1 \leq i \leq n$. Then 
\begin{enumerate}[\normalfont (i)]
\item $X^k e_i = e_{i-k}$ for all $k \geq 1$,
\item $X^k \cdot e_i^* = \sum_{i+k \leq j \leq n} (-1)^{i+j} \binom{j-i-1}{k-1} e_j^*$ for all $k \geq 1$ and $1 \leq i \leq n$. 
\end{enumerate}
\end{lemma}

\begin{proof}Statement (i) is obvious since $Xe_i = e_{i-1}$. For (ii), first note that $u^{-1} e_i = \sum_{1 \leq j \leq i} (-1)^{i+j} e_j$ for all $1 \leq i \leq n$, which is easily verified by induction on $i$. Consequently, \begin{equation}\label{eq:u_on_dual} u \cdot e_i^* = \sum_{i \leq j \leq n} (-1)^{i+j} e_j^* \end{equation} for all $1 \leq i \leq n$. We now prove (ii) by induction on $k$. The case $k = 1$ is given by~\eqref{eq:u_on_dual}. Suppose that (ii) holds for some $k \geq 1$. Then \begin{align} X^{k+1} \cdot e_i^* &= \sum_{i+k \leq j \leq n} (-1)^{i+j} \binom{j-i-1}{k-1} Xe_j^* \nonumber \\ &= \sum_{i+k \leq j \leq n} \sum_{j+1 \leq j' \leq n} (-1)^{i+j} \binom{j-i-1}{k-1} (-1)^{j+j'} e_{j'}^* \label{eq:combinestep}. \end{align} Collecting the terms in~\eqref{eq:combinestep}, we see that for $1 \leq j' \leq n$, the coefficient of $e_{j'}^*$ in $X^{k+1} \cdot e_i^*$ is zero if $j' < i+k+1$. When $j' \geq i+k+1$, the coefficient is equal to \begin{align} &(-1)^{i+j'} \sum_{i+k \leq j \leq j'-1} \binom{j-i-1}{k-1} \nonumber \\ =\ &(-1)^{i+j'} \sum_{0 \leq j \leq j'-i-k-1} \binom{j+k-1}{k-1} \nonumber \\ =\ & (-1)^{i+j'} \binom{j'-i-1}{j'-i-k-1} \label{eq:hockeystick} \\ =\ & (-1)^{i+j'} \binom{j'-i-1}{k} \label{eq:laststepsX} \end{align} where~\eqref{eq:hockeystick} is given by a standard combinatorial identity (``hockey-stick identity''), see e.g. \cite[1.2.6, (10)]{Knuth}. We conclude that $$X^{k+1} \cdot e_i^* = \sum_{i+k+1 \leq j \leq n} (-1)^{i+j} \binom{j-i-1}{k} e_j^*,$$ which completes the proof of the lemma.\end{proof}

\begin{lemma}\label{lemma:tensorpowerX}
Let $V_1$ and $V_2$ be $K[u]$-modules. Then $$X^k \cdot (v \otimes w) = \sum_{\substack{0 \leq t \leq k \\ 0 \leq s \leq t}} \binom{k}{t} \binom{t}{s} X^t v \otimes X^{k-s} w$$ for all $v \in V_1$ and $w \in V_2$ and $k \geq 1$.
\end{lemma}

\begin{proof}We proceed by induction on $k$. We have $u\cdot (v \otimes w) = uv \otimes uw = (Xv + v) \otimes (Xw + w),$ so $X \cdot (v \otimes w) = Xv \otimes Xw + Xv \otimes w + v \otimes Xw$ and thus the lemma holds when $k = 1$. Suppose that the lemma holds for some $k \geq 1$. Then \begin{equation}\label{eq:tensor_indk} X^k \cdot (v \otimes w) = \sum_{0 \leq t,s \leq k} \lambda_{t,s}^{(k)} X^t v \otimes X^s w,\end{equation} where $\lambda_{t,s}^{(k)} = \binom{k}{t} \binom{t}{k-s}$ for all $0 \leq t,s \leq k$. Applying $X$ to both sides of~\eqref{eq:tensor_indk}, we find that \begin{align*}X^{k+1} \cdot (v \otimes w) &= \sum_{0 \leq t,s \leq k} \lambda_{t,s}^{(k)} (X^{t+1}v \otimes X^{s+1}w + X^{t+1}v \otimes X^s w + X^t v \otimes X^{s+1}w) \\ &=\sum_{0 \leq t,s \leq k+1} \lambda_{t,s}^{(k+1)} X^tv \otimes X^sw\end{align*} where $\lambda^{(k+1)}_{0,0} = 0$, and \begin{align*} \lambda^{(k+1)}_{0,s} &= \lambda^{(k)}_{0,s-1}, \text{ for all } 0 < s \leq k+1, \\ \lambda^{(k+1)}_{t,0} &= \lambda^{(k)}_{t-1,0}, \text{ for all } 0 < t \leq k+1, \\ \lambda^{(k+1)}_{t,s} &= \lambda^{(k)}_{t-1,s-1} + \lambda^{(k)}_{t-1,s} + \lambda^{(k)}_{t,s-1}, \text{ for all } 0 < t,s \leq k+1.\end{align*} Now a straightforward calculation shows that $\lambda_{t,s}^{(k+1)} = \binom{k+1}{t} \binom{t}{k+1-s}$ for all $0 \leq t,s \leq k+1$, which proves the lemma.\end{proof}

For the rest of this section, set $\alpha = \nu_p(n)$. For all $0 \leq \beta \leq \alpha$, let $k_\beta \in \Z$ be such that $n = p^{\beta} k_{\beta}$, and define \begin{equation}\label{eq:deltabeta}\delta_\beta = \sum_{1 \leq i \leq p^{\beta}} \sum_{0 \leq j \leq k_{\beta}-1} (-1)^{i+1} (e_{jp^\beta + i} \otimes e_{jp^\beta + 1}^*).\end{equation} Note that $\delta_0 = \sum_{i = 1}^n e_i \otimes e_i^*$ spans the unique $1$-dimensional $\SL(V)$-submodule of $V \otimes V^*$ by Lemma \ref{lemma:typeAomegabasis}. 

The following proposition was suggested to us by calculations done for small $n$, and it is key in the proof of our main result.

\begin{prop}\label{prop:powerXondelta}
Let $1 \leq \beta \leq \alpha$. Then $X^{(p-1)p^{\beta - 1}} \cdot \delta_{\beta} = \delta_{\beta - 1}$.
\end{prop}

In order to prove Proposition \ref{prop:powerXondelta}, we will first need the following lemma, which describes the action of $X^{(p-1)p^{\beta-1}}$ on $V \otimes V^*$.

\begin{lemma}\label{lemma:powerXbeta_formula} Let $1 \leq i, j \leq n$. Then $$X^{(p-1)p^{\beta-1}} \cdot e_{i} \otimes e_j^* = \sum_{\substack{i - (p-1)p^{\beta-1} \leq t' \leq i \\ t' \equiv i \mod{p^{\beta-1}} \\ q \geq 1}} (-1)^{q+1} e_{t'} \otimes (e^*_{t' + j - i - p^{\beta-1} + qp^\beta} + e^*_{t' + j - i + qp^\beta})$$ for all $1 \leq \beta \leq \alpha$.\end{lemma} 

\begin{proof}
From Lemma \ref{lemma:tensorpowerX} and Lemma \ref{lemma:Xpower_basic} (i), we see that $X^{(p-1)p^{\beta-1}} \cdot e_{i} \otimes e_j^*$ is equal to \begin{equation}\label{eq:betafirst}\sum_{\substack{0 \leq t \leq (p-1)p^{\beta-1} \\ 0 \leq s \leq t}} \binom{(p-1)p^{\beta-1}}{t} \binom{t}{s} e_{i-t} \otimes X^{(p-1)p^{\beta-1} - s}e_j^*.\end{equation} By Lucas' theorem, we have $\binom{(p-1)p^{\beta-1}}{t} \equiv 0 \mod{p}$ if $t \not\equiv 0 \mod{p^{\beta-1}}$. If $t \equiv 0 \mod{p^{\beta-1}}$, then by Lucas' theorem and Lemma \ref{lemma:pminus1modp}, we get $\binom{(p-1)p^{\beta-1}}{t} \equiv \binom{p-1}{t/p^{\beta-1}} \equiv (-1)^{t/p^{\beta-1}} \equiv (-1)^t \mod{p}$. Therefore we can write~\eqref{eq:betafirst} as \begin{align}\sum_{\substack{0 \leq t \leq (p-1)p^{\beta-1} \\ t \equiv 0 \mod{p^{\beta-1}} \\ 0 \leq s \leq t}} & (-1)^t \binom{t}{s} e_{i-t} \otimes X^{(p-1)p^{\beta-1} - s}e_j^* \nonumber \\ = \sum_{\substack{i-(p-1)p^{\beta-1} \leq t' \leq i \\ t' \equiv i \mod{p^{\beta-1}}  \\ 0 \leq s \leq i-t'}} & (-1)^{i-t'} \binom{i-t'}{s} e_{t'} \otimes X^{(p-1)p^{\beta-1} - s}e_j^* \label{eq:betasecond}.\end{align} By Lemma \ref{lemma:Xpower_basic} (ii), we have $$X^{(p-1)p^{\beta-1} - s}e_j^* = \sum_{(p-1)p^{\beta-1} - s + j \leq s' \leq n} (-1)^{s'-j} \binom{s'-j-1}{(p-1)p^{\beta-1} - s - 1} e^*_{s'}.$$ Thus in~\eqref{eq:betasecond} the coefficient of $e_{t'} \otimes e_{s'}$ is equal to \begin{equation}\label{eq:betathird}\sum_{(p-1)p^{\beta-1} - s' + j \leq s \leq i-t'} (-1)^{i-t'+s'-j}\binom{i-t'}{s} \binom{s'-j-1}{(p-1)p^{\beta-1} - s - 1}.\end{equation} Set $\Delta = i-t'+s'-j - (p-1)p^{\beta-1}$. If $\Delta < 0$, then~\eqref{eq:betathird} is an empty sum, and thus equal to zero. Suppose next that $\Delta \geq 0$. We have $(-1)^{i-t'+s'-j} = (-1)^{\Delta}$, and by shifting the summation index by $s \mapsto s + (p-1)p^{\beta-1} - s' + j$, we see that~\eqref{eq:betathird} equals \begin{align}& (-1)^\Delta \sum_{0 \leq s \leq \Delta} \binom{i-t'}{s + (p-1)p^{\beta-1} - s' + j} \binom{s'-j-1}{s'-j-1-s} \nonumber \\ &= (-1)^{\Delta} \sum_{0 \leq s \leq \Delta} \binom{i-t'}{\Delta-s} \binom{s'-j-1}{s} \nonumber \\ &= (-1)^{\Delta} \binom{\Delta + (p-1)p^{\beta-1} - 1}{\Delta} \label{eq:chuvandermonde} \\ &= (-1)^{\Delta} \binom{\Delta + (p-1)p^{\beta-1} - 1}{(p-1)p^{\beta-1} - 1} \label{eq:betafourth} \end{align} where~\eqref{eq:chuvandermonde} is given by the Chu-Vandermonde identity, see e.g. \cite[1.2.6, (21)]{Knuth}. Write $\Delta = qp^{\beta} + r$, where $q \geq 0$ and $0 \leq r < p^{\beta}$. Since $$(p-1)p^{\beta-1} - 1 = (p-2)p^{\beta-1} + \sum_{0 \leq j \leq \beta-2} (p-1)p^j,$$ we conclude from Lucas' theorem that $\binom{\Delta + (p-1)p^{\beta-1} - 1}{(p-1)p^{\beta-1} - 1} \not\equiv 0 \mod{p}$ if and only if $r = p^{\beta-1}$ or $r = 0$. Furthermore, if $r = p^{\beta-1}$ or $r = 0$, then $\binom{\Delta + (p-1)p^{\beta-1} - 1}{(p-1)p^{\beta-1} - 1} \equiv (-1)^r \mod{p}$, so~\eqref{eq:betafourth} is equal to $(-1)^q$. 

We conclude then that for $i - (p-1)p^{\beta-1} \leq t' \leq i$ and $1 \leq s' \leq n$, the coefficient of $e_{t'} \otimes e_{s'}$ in $X^{(p-1)p^{\beta-1}} \cdot e_{i} \otimes e_j^*$ is nonzero if and only if $t' \equiv i \mod{p^{\beta-1}}$, and for $\Delta = i-t'+s'-j - (p-1)p^{\beta-1}$ we have $\Delta = qp^\beta$ or $\Delta = qp^{\beta} + p^{\beta-1}$ for some $q \geq 0$. In this case the coefficient of $e_{t'} \otimes e_{s'}$ is equal to $(-1)^q$, and $$s' = t'+j-i-p^{\beta-1} + (q+1)p^{\beta}$$ or $$s' = t'+j-i+(q+1)p^{\beta},$$ according to whether $\Delta = qp^\beta$ or $\Delta = qp^{\beta} + p^{\beta-1}$. From this conclusion, the lemma follows.\end{proof}

\begin{proof}[Proof of Proposition \ref{prop:powerXondelta}]Let $1 \leq i \leq p^{\beta}$ and $0 \leq j \leq k_\beta - 1$. Then $X^{(p-1)p^{\beta-1}} \cdot e_{jp^\beta + i} \otimes e_{jp^\beta + 1}^*$ equals \begin{equation}\label{eq:firstsum}\sum_{\substack{(j-1)p^\beta + i + p^{\beta-1} \leq t' \leq jp^\beta + i \\ t' \equiv i \mod{p^{\beta-1}} \\ q \geq 1}} (-1)^{q+1} e_{t'} \otimes (e^*_{t' - i + 1 - p^{\beta-1} + qp^\beta} + e^*_{t' - i + 1 + qp^\beta})\end{equation} by Lemma \ref{lemma:powerXbeta_formula}. Taking the sum of~\eqref{eq:firstsum} over all $0 \leq j \leq k_\beta - 1$, we get \begin{equation}\label{eq:secondsum}\sum_{\substack{-p^\beta + i + p^{\beta-1} \leq t' \leq (k_\beta - 1)p^\beta + i \\ t' \equiv i \mod{p^{\beta-1}} \\ q \geq 1}} (-1)^{q+1} e_{t'} \otimes (e^*_{t' - i + 1 - p^{\beta-1} + qp^\beta} + e^*_{t' - i + 1 + qp^\beta}).\end{equation} Note that if $t' > (k_\beta - 1)p^\beta + i$ and $t' \equiv i \mod{p^{\beta-1}}$, then for all $q \geq 1$ we have $t' - i + 1 - p^{\beta-1} + qp^\beta > n$ and thus $e^*_{t' - i + 1 - p^{\beta-1} + qp^\beta} + e^*_{t' - i + 1 + qp^\beta} = 0$. Similarly if $t' < -p^{\beta}+i+p^{\beta-1}$ and $t' \equiv i \mod{p^{\beta-1}}$, then $t' \leq 0$ and thus $e_{t'} = 0$. Therefore we can write~\eqref{eq:secondsum} as \begin{equation}\label{eq:thirdsum}z_{i+p^{\beta-1}} + z_{i},\end{equation}where we define $$z_i = \sum_{\substack{t' \equiv i \mod{p^{\beta-1}} \\ q \geq 1}} (-1)^{q+1} e_{t'} \otimes e^*_{t' - i + 1 + qp^\beta}.$$ Let $1 \leq i' \leq p^{\beta}$. Since~\eqref{eq:secondsum} and~\eqref{eq:thirdsum} are equal, we have \begin{align*} X^{(p-1)p^{\beta-1}} & \cdot \sum_{\substack{1 \leq i \leq p^{\beta} \\ i \equiv i' \mod{p^{\beta-1}}}} \sum_{0 \leq j \leq k_\beta - 1} (-1)^{i+1} e_{jp^\beta + i} \otimes e_{jp^\beta + 1}^* \\[5pt] &= \sum_{0 \leq f \leq p-1} (-1)^{i'+f} (z_{i'+(f+1)p^{\beta-1}} + z_{i' + fp^{\beta-1}}) \\[5pt] &= z_{i'} + z_{i'+p^\beta}.\end{align*} Therefore \begin{equation}\label{eq:finalsum}X^{(p-1)p^{\beta-1}} \cdot \delta_{\beta} = \sum_{1 \leq i' \leq p^{\beta-1}} (z_{i'} + z_{i'+p^\beta}).\end{equation} Finally, for all $1 \leq i' \leq p^{\beta-1}$ we have \begin{align*}z_{i'} + z_{i'+p^{\beta}} &= z_{i'} + \sum_{\substack{t' \equiv i' \mod{p^{\beta-1}} \\ q \geq 1}} (-1)^{q+1} e_{t'} \otimes e^*_{t' - i' + 1 + (q-1)p^\beta} \\ &= z_{i'} + \left( -z_{i'} + \sum_{t' \equiv i' \mod{p^{\beta}}} e_{t'} \otimes e^*_{t'-i'+1} \right) \\ &= \sum_{0 \leq j \leq k_{\beta-1}-1} e_{jp^{\beta-1} + i'} \otimes e^*_{jp^{\beta-1}+1}\end{align*} which together with~\eqref{eq:finalsum} proves that $X^{(p-1)p^{\beta-1}} \cdot \delta_{\beta} = \delta_{\beta-1}$.\end{proof}

\begin{seur}\label{corollary:mainresult}
Let $1 \leq \beta \leq \alpha$. Then $X^{p^{\beta} - 1} \cdot \delta_\beta = \delta_0$.
\end{seur}

\begin{proof}Since $p^{\beta} - 1 = \sum_{0 \leq j \leq \beta - 1} (p-1)p^j$, the claim is immediate from Proposition \ref{prop:powerXondelta}.\end{proof}

\section{Main results}\label{section:mainresult}

In this section, we prove our main results. Below for a unipotent element $u \in \SL(V)$, we use the notation $V_i$ for indecomposable $K[u]$-modules as in Section \ref{section:tensordecomp}. Note that for $G = \SL(V)$, the $G$-conjugacy class of $u \in G$ is determined by the decomposition of $V \downarrow K[u]$ into indecomposable summands. By \cite[Proposition 2 of Chapter II]{Gerstenhaber}, the same is also true for $G = \Sp(V)$ and $G = \operatorname{O}(V)$ if $p > 2$.

\begin{lause}\label{thm:mainthmA}
Let $G = \SL(V)$, where $\dim V = n$ for some $n \geq 2$. Let $u \in G$ be a unipotent element and $V \downarrow K[u] = V_{d_1} \oplus \cdots \oplus V_{d_t}$, where $t \geq 1$ and $d_r \geq 1$ for all $1 \leq r \leq t$. Set $\alpha = \nu_p(\gcd(d_1, \ldots, d_t))$. Let $u_0$ be the action of $u$ on $V \otimes V^*$, and let $u_0''$ be the action of $u$ on $L_G(\varpi_1 + \varpi_{n-1})$. Then the Jordan block sizes of $u_0''$ are determined from those of $u_0$ as follows:

\begin{enumerate}[\normalfont (i)]
\item If $p \nmid n$, then $r_1(u_0'') = r_1(u_0) - 1$ and $r_m(u_0'') = r_m(u_0)$ for all $m \neq 1$.
\item If $p \mid n$ and $\alpha = 0$, then $r_1(u_0'') = r_1(u_0) - 2$ and $r_m(u_0'') = r_m(u_0)$ for all $m \neq 1$.
\item If $p \mid n$ and $\alpha > 0$:
	\begin{enumerate}[\normalfont (a)]
	\item If $p \mid \frac{n}{p^{\alpha}}$, then $r_{p^{\alpha}}(u_0'') = r_{p^{\alpha}}(u_0) - 2$, $r_{p^{\alpha}-1}(u_0'') = 2$ and $r_m(u_0'') = r_m(u_0)$ for all $m \neq p^{\alpha}, p^{\alpha} - 1$.
	\item If $p \nmid \frac{n}{p^{\alpha}}$ and $p^{\alpha} > 2$, then $r_{p^{\alpha}}(u_0'') = r_{p^{\alpha}}(u_0) - 1$, $r_{p^{\alpha}-2}(u_0'') = 1$ and $r_m(u_0'') = r_m(u_0)$ for all $m \neq p^{\alpha}, p^{\alpha} - 2$.
	\item If $p \nmid \frac{n}{p^{\alpha}}$ and $p^{\alpha} = 2$, then $r_{2}(u_0'') = r_{2}(u_0) - 1$ and $r_m(u_0'') = r_m(u_0)$ for all $m \neq 2$.
	\end{enumerate}
\end{enumerate}
\end{lause}

\begin{proof}If $p \nmid n$, then $V \otimes V^* \cong L_G(\varpi_1 + \varpi_{n-1}) \oplus L_G(0)$ by Lemma \ref{lemma:typeAomega}, so it is obvious that the theorem holds in this case. Suppose then for the rest of the proof that $p \mid n$.

We first need to set up a suitable basis for $V$ and some notation. Let $V = W_1 \oplus \cdots \oplus W_t$ such that $W_r$ are $u$-invariant and $W_r \cong V_{d_r}$ for all $1 \leq r \leq t$. For each $r$, choose a basis $(e_j^{(r)})_{1 \leq j \leq d_{r}}$ of $W_r$ such that $ue_j^{(r)} = e_j^{(r)} + e_{j-1}^{(r)}$ for all $1 \leq j \leq d_r$, where we set $e_j^{(r)} = 0$ for all $j \leq 0$. For the basis $(e_j^{(r)})$ of $V$, we let $(e_j^{(r)^*})$ be the corresponding dual basis of $V^*$. Next, let $\varphi: V \otimes V^* \rightarrow K$ be the surjective morphism of $G$-modules (where $G$ acts trivially on $K$) defined by $\varphi(v \otimes f) = f(v)$. (With the usual isomorphism $V \otimes V^* \cong \operatorname{End}(V)$, this map corresponds to the trace map $\operatorname{End}(V) \rightarrow K$.) For all $1 \leq r \leq t$, we let $\delta_0^{(r)}$ be the element $$\sum_{1 \leq j \leq d_r} e_j^{(r)} \otimes  e_j^{(r)^*}$$ of $V \otimes V^*$. Note that $\delta_0 = \sum_{1 \leq r \leq t} \delta_0^{(r)}$ spans the unique $1$-dimensional $G$-submodule of $V \otimes V^*$ by Lemma \ref{lemma:typeAomegabasis}.  Furthermore, we have $\delta_0 \in \operatorname{Ker} \varphi$ since $p \mid n$, so \begin{equation}\label{eq:omega1omegal}\operatorname{Ker} \varphi / \langle \delta_0 \rangle \cong L_G(\varpi_1 + \varpi_{n-1})\end{equation} by Lemma \ref{lemma:typeAomega}. We denote by $u_0'$ the action of $u_0$ on $V \otimes V^* / \langle \delta_0 \rangle$, and by~\eqref{eq:omega1omegal} we can assume without loss of generality that $u_0''$ is the action of $u_0$ induced on $\operatorname{Ker} \varphi / \langle \delta_0 \rangle$. Let $X$ be the element $u-1$ of $K[u]$.

Our basic strategy in the proof of the theorem consists of two applications of Lemma \ref{jordanrestriction}. First we apply it to $u_0$ to find the Jordan block sizes of $u_0'$, and later we shall apply it to $u_0'$ to find the Jordan block sizes of $u_0''$.

As the first step of the actual proof, we will show that \begin{equation}\label{eq:firstrestriction}X^{p^{\alpha}} \cdot v = 0 \text{ for some } v \in V \otimes V^* - \operatorname{Ker} \varphi.\end{equation} To find such a $v$, choose a $1 \leq r' \leq t$ such that $\nu_p(d_{r'}) = \alpha$, so $d_{r'} = p^{\alpha} k_{r'}$ where $p \nmid k_{r'}$. Now $(u-1) e_{j}^{(r')} = e_{j-1}^{(r')}$ for all $j$, so $(u^{p^{\alpha}} - 1)e_{j}^{(r')} = (u-1)^{p^{\alpha}} e_{j}^{(r')} = e_{j-p^{\alpha}}^{(r')}$ for all $j$. Therefore the subspace $\langle e_{1 + jp^{\alpha}}^{(r')} : 0 \leq j \leq k_{r'}-1 \rangle$ is $u^{p^{\alpha}}$-invariant, so by Lemma \ref{lemma:typeAomega} the vector $$v = \sum_{0 \leq j \leq k_{r'}-1} e_{1 + jp^{\alpha}}^{(r')} \otimes {e_{1 + jp^{\alpha}}^{(r')^*}}$$ is fixed by $u^{p^{\alpha}}$. Hence $X^{p^{\alpha}} \cdot v = 0$, and $v \not\in \operatorname{Ker} \varphi$ since $\varphi(v) = k_{r'}$. This proves~\eqref{eq:firstrestriction}.

Since $V \otimes V^*$ is self-dual and uniserial by Lemma \ref{lemma:typeAomega}, we have $(\operatorname{Ker} \varphi)^* \cong V \otimes V^* / \langle \delta_0 \rangle$ as $G$-modules, so \begin{equation}\label{eq:dualityequation}\operatorname{Ker} \varphi \cong V \otimes V^* / \langle \delta_0 \rangle\end{equation} as $K[u]$-modules. Thus if $\alpha = 0$, then~\eqref{eq:firstrestriction} and Lemma \ref{jordanrestriction} show that $r_1(u_0') = r_1(u_0) - 1$ and $r_m(u_0') = r_m(u_0)$ for all $m \neq 1$. Furthermore, we have $\operatorname{Ker}(u_0' - 1) \not\subset \operatorname{Ker} \varphi / \langle \delta_0 \rangle$ by~\eqref{eq:firstrestriction}. Thus Lemma \ref{jordanrestriction} gives $r_1(u_0'') = r_1(u_0) - 2$ and $r_m(u_0'') = r_m(u_0)$ for all $m \neq 1$, proving the theorem in this case.

Next we consider the case where $\alpha > 0$. The smallest Jordan block size of $u_0$ is $p^{\alpha}$ by Lemma \ref{lemma:smallestontsquare}, so \begin{equation}\label{eq:keralpha1} \operatorname{Ker}(u_0 - 1)^{p^{\alpha}-1} \subset \operatorname{Ker} \varphi \end{equation} by Lemma \ref{jordanrestriction}. Thus~\eqref{eq:firstrestriction} and Lemma \ref{jordanrestriction}, along with~\eqref{eq:dualityequation}, gives $r_{p^{\alpha}-1}(u_0') = 1$, $r_{p^{\alpha}}(u_0') = r_{p^{\alpha}}(u_0) - 1$, and $r_{m}(u_0') = r_m(u_0)$ for all $m \neq p^\alpha, p^{\alpha}-1$. Note that \begin{equation}\label{eq:alpha}\operatorname{Ker}(u_0' - 1)^{p^{\alpha}} \not\subset \operatorname{Ker} \varphi / \langle \delta_0 \rangle\end{equation} by~\eqref{eq:firstrestriction}. Furthermore, since $p^{\alpha}-1$ is the smallest Jordan block size of $u_0'$, we have \begin{equation}\label{eq:alpham2}\operatorname{Ker}(u_0' - 1)^{p^{\alpha}-2} \subset \operatorname{Ker} \varphi / \langle \delta_0 \rangle\end{equation} by Lemma \ref{jordanrestriction}. Hence to describe the Jordan block sizes of $u_0''$ in terms of the Jordan block sizes of $u_0'$ using Lemma \ref{jordanrestriction}, it remains to determine when \begin{equation}\label{eq:finalquestion}\operatorname{Ker}(u_0' - 1)^{p^{\alpha}-1} \subset \operatorname{Ker} \varphi / \langle \delta_0 \rangle.\end{equation} 

We shall show that~\eqref{eq:finalquestion} holds if and only if $p \mid \frac{n}{p^{\alpha}}$, which together with~\eqref{eq:alpha}, \eqref{eq:alpham2}, and Lemma \ref{jordanrestriction} completes the proof of the theorem. To this end, write $d_r = p^{\alpha}k_r$ and set $$\delta_r = \sum_{1 \leq i \leq p^{\alpha}} \sum_{0 \leq j \leq k_r - 1} (-1)^{i+1} (e_{jp^\alpha + i}^{(r)} \otimes e_{jp^\alpha + 1}^{(r)^*})$$ for all $1 \leq r \leq t$. Note that the action of $u$ on $W_r$ is defined as in Section \ref{section:calculations}, and $\delta_r$ is defined as in~\eqref{eq:deltabeta}, so it follows from Corollary \ref{corollary:mainresult} that $X^{p^{\alpha}-1} \cdot \delta_r = \delta_0^{(r)}$ for all $1 \leq r \leq t$. Thus for $\delta = \sum_{1 \leq r \leq t} \delta_r$ we have $X^{p^{\alpha}-1} \cdot \delta = \delta_0$. Any solution to $X^{p^{\alpha}-1} \cdot z = \delta_0$ is equal to $\delta$ modulo $\operatorname{Ker}(u_0 - 1)^{p^{\alpha} - 1}$, so we conclude that $$\operatorname{Ker}(u_0' - 1)^{p^{\alpha} - 1} = \left( \operatorname{Ker}(u_0 - 1)^{p^{\alpha} - 1} + \langle \delta \rangle \right) / \langle \delta_0 \rangle.$$ Then from~\eqref{eq:keralpha1}, we see that~\eqref{eq:finalquestion} holds if and only if $\delta \in \operatorname{Ker} \varphi$. Now $\varphi(\delta_r) = k_r$, so $\varphi(\delta) = k_1 + \cdots + k_r = \frac{n}{p^{\alpha}}$, and consequently~\eqref{eq:finalquestion} holds if and only if $p \mid \frac{n}{p^{\alpha}}$.\end{proof}

\begin{seur}\label{corollary:SPmain}Assume $p > 2$, and let $G = \Sp(V)$, where $\dim V = n$ for some $n \geq 4$. Let $u \in G$ be a unipotent element and $V \downarrow K[u] = V_{d_1} \oplus \cdots \oplus V_{d_t}$, where $t \geq 1$ and $d_r \geq 1$ for all $1 \leq r \leq t$. Set $\alpha = \nu_p(\gcd(d_1, \ldots, d_t))$. Let $u_0$ be the action of $u$ on $\wedge^2(V)$, and let $u_0''$ be the action of $u$ on $L_G(\varpi_2)$. Then the Jordan block sizes of $u_0''$ are determined from those of $u_0$ by the rules (i) -- (iii) of Theorem \ref{thm:mainthmA}.\end{seur}

\begin{proof}We begin by showing that we have isomorphisms \begin{equation}\label{eq:firstselfdual}V \otimes V^* \cong \wedge^2(V) \oplus S^2(V)\end{equation} and \begin{equation}\label{eq:secondselfdual}L_{\SL(V)}(\varpi_1 + \varpi_{n-1}) \cong L_G(\varpi_2) \oplus S^2(V)\end{equation} of $G$-modules. The isomorphism~\eqref{eq:firstselfdual} follows since $V \cong V^*$, and furthermore since $V \otimes V \cong \wedge^2(V) \oplus S^2(V)$ when $p > 2$. All the trivial $G$-composition factors of $V \otimes V^* \cong \wedge^2(V) \oplus S^2(V)$ lie in the $\wedge^2(V)$ summand by Lemma \ref{lemma:typeComega}, so~\eqref{eq:secondselfdual} follows from Lemma \ref{lemma:typeAomega} and Lemma \ref{lemma:typeComega}.

By~\eqref{eq:firstselfdual} and~\eqref{eq:secondselfdual}, the Jordan block sizes of $u$ on $\wedge^2(V)$ and $L_G(\varpi_2)$ differ in the same way as the Jordan block sizes of $u$ on $V \otimes V^*$ and $L_{\SL(V)}(\varpi_1 + \varpi_{n-1})$. Thus the result follows from Theorem \ref{thm:mainthmA}.
\end{proof}

\begin{seur}\label{corollary:SOmain}Assume $p > 2$, and let $G = \SO(V)$, where $\dim V = n$ for some $n \geq 5$. Let $u \in G$ be a unipotent element and $V \downarrow K[u] = V_{d_1} \oplus \cdots \oplus V_{d_t}$, where $t \geq 1$ and $d_r \geq 1$ for all $1 \leq r \leq t$. Set $\alpha = \nu_p(\gcd(d_1, \ldots, d_t))$. Let $u_0$ be the action of $u$ on $S^2(V)$, and let $u_0''$ be the action of $u$ on $L_G(2\varpi_1)$. Then the Jordan block sizes of $u_0''$ are determined from those of $u_0$ by the rules (i) -- (iii) of Theorem \ref{thm:mainthmA}.\end{seur}

\begin{proof}Similarly to the proof of Corollary \ref{corollary:SPmain}, the result follows using the isomorphisms $$V \otimes V^* \cong \wedge^2(V) \oplus S^2(V)$$ and $$L_{\SL(V)}(\varpi_1 + \varpi_{n-1}) \cong \wedge^2(V) \oplus L_G(2 \varpi_1)$$ of $G$-modules, see Lemma \ref{lemma:typeBDomega}.\end{proof}

\clearpage
\begin{esim}Let $G = \SL(V)$ with $\dim V = n$ for $n \geq 2$. In Table \ref{table:examples} below, we give for all $2 \leq n \leq 6$ and all unipotent elements $u \in G$ the Jordan normal form of the action of $u$ on $V \otimes V^*$ and $L_G(\omega_1 + \omega_{n-1})$, in the case where $p \mid n$. These examples illustrate all of the cases (ii), (iii)(a) -- (c) of Theorem \ref{thm:mainthmA}. In the table, we use the notation $d_1^{n_1}, \ldots, d_t^{n_t}$ for a $K[u]$-module of the form $n_1 \cdot V_{d_1} \oplus \cdots \oplus n_t \cdot V_{d_t}$, where $0 < d_1 < \cdots < d_t$ and $n_i \geq 1$ for all $1 \leq i \leq t$.\end{esim}

\begin{table}[!htbp]
\centering
\caption{}\label{table:examples}
%\caption{For $G = \SL(V)$ with $\dim V = n$ for $p \mid n$ and $3 \leq n \leq 6$, Jordan block sizes of unipotent elements $u \in G$ on $V \otimes V^*$ and $L_{G}(\varpi_1 + \varpi_{n-1})$. Here we write $d_1^{n_1}, \ldots, d_t^{n_t}$ for a $K[u]$-module of the form $n_1 \cdot V_{d_1} \oplus \cdots \oplus n_t \cdot V_{d_t}$, where $0 < d_1 < \cdots < d_t$ and $n_i \geq 1$ for all $1 \leq i \leq t$.}\label{table:examples}
\footnotesize
%\scriptsize
\begin{tabular}{| c | l | l | l |}
\hline
$G$              & $V \downarrow K[u]$       & $V \otimes V^* \downarrow K[u]$         & $L_G(\omega_1 + \omega_{n-1}) \downarrow K[u]$ \\ \hline
$n = 2$, $p = 2$ & $2$ & $2^2$ & $2$ \\
                 & $1^2$ & $1^4$ & $1^2$ \\
& & & \\
$n = 3$, $p = 3$ & $3$        & $3^{3}$               & $1, 3^{2}$ \\
                 & $1, 2$     & $1^{2}, 2^{2}, 3$     & $2^{2}, 3$ \\
                 & $1^{3}$    & $1^{9}$               & $1^{7}$ \\
								 & & & \\
$n = 4$, $p = 2$ & $4$        & $4^{4}$               & $2, 4^{3}$ \\
                 & $1, 3$     & $1^{2}, 3^{2}, 4^{2}$ & $3^{2}, 4^{2}$ \\
                 & $2^{2}$    & $2^{8}$               & $1^{2}, 2^{6}$ \\
                 & $1^{2}, 2$ & $1^{4}, 2^{6}$        & $1^{2}, 2^{6}$ \\
                 & $1^{4}$    & $1^{16}$              & $1^{14}$ \\
								 & & & \\
$n = 5$, $p = 5$ & $5$        & $5^{5}$ & $3, 5^{4}$ \\
                 & $1, 4$     & $1^{2}, 4^{2}, 5^{3}$ & $4^{2}, 5^{3}$ \\
                 & $2, 3$     & $1^{2}, 2^{2}, 3^{2}, 4^{2}, 5$ & $2^{2}, 3^{2}, 4^{2}, 5$ \\
                 & $1^{2}, 3$ & $1^{5}, 3^{5}, 5$ & $1^{3}, 3^{5}, 5$ \\
                 & $1, 2^{2}$ & $1^{5}, 2^{4}, 3^{4}$ & $1^{3}, 2^{4}, 3^{4}$ \\
                 & $1^{3}, 2$ & $1^{10}, 2^{6}, 3$ & $1^{8}, 2^{6}, 3$ \\
                 & $1^{5}$    & $1^{25}$ & $1^{23}$ \\
								 & & & \\
$n = 6$, $p = 2$ & $6$ & $2^{2}, 8^{4}$ & $2, 8^{4}$ \\
                 & $1, 5$ & $1^{2}, 4^{2}, 5^{2}, 8^{2}$ & $4^{2}, 5^{2}, 8^{2}$ \\
                 & $2, 4$ & $2^{2}, 4^{8}$ & $2, 4^{8}$ \\
                 & $1^{2}, 4$ & $1^{4}, 4^{8}$ & $1^{2}, 4^{8}$ \\
                 & $3^{2}$ & $1^{4}, 4^{8}$ & $1^{2}, 4^{8}$ \\
                 & $1, 2, 3$ & $1^{2}, 2^{6}, 3^{2}, 4^{4}$ & $2^{6}, 3^{2}, 4^{4}$ \\
                 & $1^{3}, 3$ & $1^{10}, 3^{6}, 4^{2}$ & $1^{8}, 3^{6}, 4^{2}$ \\
                 & $2^{3}$ & $2^{18}$ & $2^{17}$ \\
                 & $1^{2}, 2^{2}$ & $1^{4}, 2^{16}$ & $1^{2}, 2^{16}$ \\
                 & $1^{4}, 2$ & $1^{16}, 2^{10}$ & $1^{14}, 2^{10}$ \\
                 & $1^{6}$ & $1^{36}$ & $1^{34}$ \\
								 & & & \\
$n = 6$, $p = 3$ & $6$ & $3^{3}, 9^{3}$ & $1, 3^{2}, 9^{3}$ \\
                 & $1, 5$ & $1^{2}, 3, 5^{3}, 7, 9$ & $3, 5^{3}, 7, 9$ \\
                 & $2, 4$ & $1^{2}, 3^{4}, 5^{3}, 7$ & $3^{4}, 5^{3}, 7$ \\
                 & $1^{2}, 4$ & $1^{5}, 3, 4^{4}, 5, 7$ & $1^{3}, 3, 4^{4}, 5, 7$ \\
                 & $3^{2}$ & $3^{12}$ & $1, 3^{11}$ \\
                 & $1, 2, 3$ & $1^{2}, 2^{2}, 3^{10}$ & $2^{2}, 3^{10}$ \\
                 & $1^{3}, 3$ & $1^{9}, 3^{9}$ & $1^{7}, 3^{9}$ \\
                 & $2^{3}$ & $1^{9}, 3^{9}$ & $1^{7}, 3^{9}$ \\
                 & $1^{2}, 2^{2}$ & $1^{8}, 2^{8}, 3^{4}$ & $1^{6}, 2^{8}, 3^{4}$ \\
                 & $1^{4}, 2$ & $1^{17}, 2^{8}, 3$ & $1^{15}, 2^{8}, 3$ \\
                 & $1^{6}$ & $1^{36}$ & $1^{34}$ \\
\hline
\end{tabular}

\end{table}

\begin{remark}Suppose that $p = 2$. Let $G = \Sp(V)$ with $\dim V = 2l$, and let $u \in G$ be a unipotent element. For example by \cite[Lemma 4.8.2]{McNinch}, we have $$\wedge^2(V) \cong \begin{cases} L_G(\varpi_2) \oplus L_G(0) &\mbox{if } l \mbox{ is odd,} \\ L_G(0) | L_G(\varpi_2) | L_G(0) \text{ } & \mbox{if } l \mbox{ is even.} \end{cases}$$ What are the Jordan block sizes of $u$ acting on $L_G(\varpi_2)$? Here one could hope for an answer in terms of the Jordan block sizes of $u$ acting on $\wedge^2(V)$, which are known \cite[Theorem 2]{GowLaffey}. 

We have not included results on this problem in this paper, but shall note the following example which demonstrates that the situation here is slightly more involved than in characteristic $p > 2$. Suppose that $\dim V = 4$. In this case, there are two conjugacy classes of unipotent elements of $G$ which act on $V$ with Jordan form $2 \cdot V_2$. It is easy to see that unipotent elements in both conjugacy classes must act on $\wedge^2(V)$ with Jordan form $2 \cdot V_1 \oplus 2 \cdot V_2$. However, a computation shows that elements from one of the classes act on $L_G(\varpi_2)$ with Jordan form $2 \cdot V_2$, while elements of the other unipotent class act on $L_G(\varpi_2)$ with Jordan form $2 \cdot V_1 \oplus V_2$.\end{remark}

\section{Acknowledgements}

%Some of the results in this paper were obtained during my doctoral studies at \'Ecole Polytechnique F\'ed\'erale de Lausanne, supported by a grant from the Swiss National Science Foundation (grant number $200021 \_ 146223$). I am very grateful to my advisor, Donna Testerman, for many helpful comments and discussions. I would also like to thank Gunter Malle for pointing out some misprints in an earlier version of this paper.

%Some of the results in this paper were obtained during my doctoral studies at \'Ecole Polytechnique F\'ed\'erale de Lausanne, supported by a grant from the Swiss National Science Foundation (grant number $200021 \_ 146223$). 

I am very grateful to Donna Testerman for many helpful comments and discussions. I would also like to thank Gunter Malle for pointing out some misprints in an earlier version of this paper.

\providecommand{\bysame}{\leavevmode\hbox to3em{\hrulefill}\thinspace}
\providecommand{\MR}{\relax\ifhmode\unskip\space\fi MR }
\providecommand{\MRhref}[2]{%
  \href{http://www.ams.org/mathscinet-getitem?mr=#1}{#2}
}
\providecommand{\href}[2]{#2}

\end{document}